\documentclass[10pt, a4paper]{article}
\usepackage{mathrsfs}
\usepackage{amsmath}
\usepackage{amssymb}
\usepackage[amsmath,thmmarks,hyperref]{ntheorem}
\usepackage[all]{xy}
\usepackage{stmaryrd}
\usepackage[bookmarksnumbered]{hyperref}
\usepackage{enumitem}
\usepackage{eucal}
\usepackage{ulem}
\usepackage{todonotes}
\usepackage{rotating}
\usepackage{appendix}
\usepackage{parskip}
\usepackage[textwidth=15cm, textheight=24cm]{geometry}
\usepackage{pdflscape}
\usepackage{setspace}
\usepackage{styledual}

\title{A monoidal algebraic model for rational $SO(2)$-spectra}
\author{David Barnes}

\begin{document}
\maketitle
\normalem
\begin{abstract}
\noindent
The category of rational $SO(2)$--equivariant spectra
admits an algebraic model. That is, there is an
abelian category $\acal(\torus)$ whose derived 
category is equivalent to the homotopy
category of rational $SO(2)$--equivariant spectra.
An important question is: does this algebraic model capture the
smash product of spectra?

The category $\acal(\torus)$ is known as Greenlees'
standard model, it is an abelian category
that has no projective objects and is constructed from
modules over a non--Noetherian ring.
Hence the standard techniques for constructing a
monoidal model structure cannot be applied.
In this paper we construct a monoidal model structure on
$\acal(\torus)$ and show that the derived product on the homotopy
category is compatible with the smash product of spectra.
The method used is related to techniques developed
by the author in earlier joint work with Roitzheim. That work
constructed a monoidal model structure on Franke's
exotic model for the $K_{(p)}$--local stable homotopy category.

We also provide a monoidal Quillen equivalence
to a simpler monoidal model category $R_\bullet \leftmod$ that has
explicit generating sets. Having monoidal model structures
on $\acal(\torus)$ and $R_\bullet \leftmod$
removes a serious obstruction to constructing a series
of monoidal Quillen equivalences between the algebraic model
and rational $SO(2)$--equivariant spectra.
\end{abstract}

\tableofcontents

\section{Introduction}

A particularly useful technique in algebraic topology has been to
construct algebraic models for stable homotopy categories.
The first example is that the homotopy category of rational spectra is
equivalent to the category of graded rational vector spaces.
More interesting examples include the work of Franke modelling
the $K_{(p)}$--local stable homotopy category  (see Roitzheim \cite{Roi08}),
work of Bousfield \cite{bous90} on $K$-local spectra
and work of Greenlees and others in the case of rational $G$-spectra:
\cite{gre99}, \cite{greo2}, \cite{greshi}
or \cite{barnesfinite}. In particular the work on rational $G$-spectra
provides classifications of rational $G$-equivariant cohomology
theories in terms of the algebraic models.

An important (and difficult) question is whether these algebraic models capture
the monoidal products. That is,
does the derived smash product in the topological setting correspond
to a derived tensor product coming from the algebraic model?
This is true in the case of rational spectra: Shipley \cite{shiHZ}
shows that commutative $\h \qq$-algebras are rational commutative
differential graded algebras.
Conversely, the author and Roitzheim \cite{brfranke} show that this is false in the
case of Franke's exotic model, even though the algebraic model
does capture the Picard group.
In this paper we focus on algebraic models for rational $G$-spectra
where the group of equivariance is $SO(2)$.

The homotopy category of rational $\torus$--equivariant spectra
is equivalent to (the derived category of) an abelian category $\acal(\torus)$
known as Greenlees's standard model.
This algebraic model is quite straightforward,
we may (very roughly) describe the objects as morphisms
of $R$-modules $\beta \co N \to R[S^{-1}] \otimes U$, where
$R$ is a commutative ring, $U$ is a $\qq$-module
and $\beta$ is an isomorphism after inverting the set $S \subset R$.
For full details see Definition \ref{def:algmodel}.
It is easy to construct objects
in $\acal(\torus)$ and calculate maps between them.
However this category exhibits some curious behaviours:
it has no projectives, limits are
complicated to construct and most functors in to the category are right adjoints.
In particular the obvious evaluation functors
(which send an object $\beta \co N \to R[S^{-1}] \otimes U$ of $\acal(\torus)$
to either the $R$-module $N$ or the $\qq$-module $U$) are left adjoints,
as is discussed at the end of Section \ref{sec:monoidal}.
Furthermore, the ring $R$ is not Noetherian and
the condition that $\beta$ be an isomorphism after inverting $S$
makes it hard to relate $\acal(\torus)$ to the category of $R$-modules.
These problems make it very difficult to construct a derived monoidal product
or a monoidal model structure where the weak equivalences
are the homology isomorphisms.
A model structure for $\acal(\torus)$ is given in \cite{gre99}. However, it is known
that this model structure cannot be monoidal, leaving
the important question of monoidality open.

In this paper we apply the methods
of Barnes and Roitzheim \cite{brfranke}
to resolve this problem and give a monoidal model
structure for $\acal(\torus)$.
By extensively studying the dualisable
objects of the category $\acal(\torus)$ we show that they can be used
to construct a new monoidal model structure.
Furthermore, the weak equivalences of this new model
structure are the homology isomorphisms,
see Theorems \ref{thm:modelstructure} and \ref{thm:monoidal}.
This model structure is Quillen
equivalent to that of Greenlees,
hence we have the correct homotopy category.
Furthermore, the induced derived monoidal product on the homotopy
category is the correct one, in the sense that it is compatible with the
short exact sequence of Theorem \ref{thm:greenleesmonoidal}.

While one could try to use the flat objects to make a
monoidal model structure, these are harder
to identify, as the ring $R$ is poorly behaved.
We also place a monoidal model structure on a larger category
$\hat{\acal}(\torus)$, whose objects are morphisms
of $R$-modules $\beta \co N \to R[S^{-1}] \otimes U$, where
$R$ is a commutative ring, $U$ is a $\qq$-module and there is no
isomorphism condition on $\beta$.  In this category there
aren't enough dualisable objects,  so instead
we use a set of flat objects
that one cannot construct in $\acal(\torus)$,
see Remark \ref{rmk:flatmodel}.

We end the paper by producing a symmetric monoidal
Quillen equivalence between $\acal(\torus)$
and a simpler  (but much larger) category
$R_\bullet \leftmod$, see Theorem \ref{thm:monoidalequivalence}.
While the weak equivalences of the
model structure we put on this larger category
are more complicated, we can give explicit
generating sets for the model structure.
It is easier still to construct objects in $R_\bullet \leftmod$
as there are are plenty of left adjoints into the category
and limits are much simpler to construct.
Moreover, $R_\bullet \leftmod$ appears
in the preprint of Greenlees and Shipley \cite{greshi}
as part of a series of Quillen equivalences between
rational $\torus$--equivariant spectra
and $\acal(\torus)$.
Hence this paper fixes the primary obstruction
to constructing a series of \emph{monoidal}
Quillen equivalences between
rational $\torus$--equivariant spectra
and $\acal(\torus)$. Such a series of Quillen equivalences
would provide a classification of ring spectra (and
modules over them) in terms of ring objects in
$\acal(\torus)$ (and modules over them). In terms
of cohomology theories, this would provide a classification
of rational $SO(2)$--equivariant cohomology theories with a
cup product.
As a further application, we
include a conjecture about extending the results
of this paper to the case of the $r$-fold product of copies of $\torus$.

\paragraph{Organisation}

In the first half we introduce the algebraic model
$\acal(\torus)$ and its properties.
Section \ref{sec:the model} has the formal definition of
$\acal(\torus)$ from \cite{gre99}.
In Section \ref{sec:limits} we define limits in $\acal(\torus)$,
introduce the larger category $R_\bullet \leftmod$
and give the adjunction between $R_\bullet \leftmod$ and $\acal(\torus)$.
We then use this adjunction in Section \ref{sec:monoidal}
to show that $\acal(\torus)$
has a closed symmetric monoidal product.

In the second half of the paper we construct the monoidal model structure.
to do so, we need a class of objects which behave well with respect to the
monoidal product. This class is introduced and
studied in Section \ref{sec:duals}.
We construct a monoidal model structure on $\acal(\torus)$ in
Section \ref{sec:dualisable} and show
that it is Quillen equivalent to
the original model structure of Greenlees
and has the correct monoidal behaviour.
Finally in Section \ref{sec:quillen}
we give the monoidal Quillen equivalence between
$\acal(\torus)$ and an explicitly defined model structure on
$R_\bullet \leftmod$.

\paragraph{Acknowledgements}
The author would like to thank Rosona Eldred, John Greenlees,
Magdalena Kedziorek and Constanze Roitzheim
for numerous helpful discussions and suggestions.
The author
gratefully acknowledges support from EPSRC grant EP/H026681/2.

\section{The model \texorpdfstring{$\acal(\torus)$}{A(T)}}\label{sec:the model}
In this section we introduce Greenlees' standard model $\acal(\torus)$.
This is the algebraic model for rational $SO(2)$--equivariant spectra.
We explain how to turn this into a differential graded
category and define the
injective model structure. The material of this section is taken from
\cite{gre99}.
We begin with the ring $\ocal_\fcal$ and the set of
``Euler classes''. The category $\acal(\torus)$ will be built from these
constructions.

\begin{definition}
Let $\fcal$ be the set of finite subgroups of $\torus$
(the cyclic groups $C_n$ for $n \geqslant 1$).
Let $\ocal_\fcal$ be the \emph{graded} ring $\prod_{H \in \fcal} \qq [c_H]$
(a countably infinite product of polynomial rings on one generator)
with $c_H$ of degree $-2$.

Define $e_H \in \ocal_\fcal$ to be the idempotent arising from
projection onto factor $H$.
In general, if $\phi$ is a subset of $\fcal$ we
define $e_\phi$ to be the idempotent
coming from projection onto the factors in $\phi$.
We let $c$ be the unique element of $\ocal_\fcal$ which in factor
$H$ is $c_H$. We can then write $c_H = e_H c$.
\end{definition}

\begin{definition}
Let $\nu \co \fcal \to \zz_{\geqslant 0}$ be a function with finite support,
Let $n$ be the maximum value of $\nu$ and partition
$\fcal$ into $n+1$ sets $\fcal_0$, $\fcal_1$, \dots $\fcal_n$,
where $H$ is in $\fcal_i$ if and only if $\nu(H) = i$.
If $N$ is an $\ocal_\fcal$--module, then we define
$\ocal_\fcal$--modules
\[
\Sigma^\nu N = \oplus_{i=0}^n \Sigma^{2i} e_{\fcal_i} N
\quad \textrm{and} \quad
\Sigma^{-\nu} N = \oplus_{i=0}^n \Sigma^{-2i} e_{\fcal_i} N.
\]
Furthermore we have a map of $\ocal_\fcal$--modules
\[
c^\nu \co N \to \Sigma^{\nu} N.
\]
We define this map differently on each of the subdivisions of
$\fcal$. On part $\fcal_i$ we use the map
$c^i \co e_{\fcal_i} N \to \Sigma^{2i} e_{\fcal_i} N$.
\end{definition}

We note that $c^\nu$ is not an element of $\ocal_\fcal$ because it does not
have the same degree in each factor $\qq[c_H]$. For the sake of expediency,
we shall often ignore this and pretend that $c^\nu$ is in $\ocal_\fcal$.
Hence $c^\nu x$ is shorthand for the image of $x$ in $\Sigma^{\nu} N$
under the map $c^\nu$.

\begin{definition}\label{def:euler}
For a function $\nu \co \fcal \to \zz_{\geqslant 0}$ with finite support,
define $c^\nu \in \ocal_\fcal$ to be the (inhomogeneous) element
satisfying $e_H c^\nu = c_H^{\nu(H)}$.
We define
\[
\ecal = \{ c^\nu \ | \ \nu \co \fcal \to \zz_{\geqslant 0} \textrm{ with finite support}  \}
\]
and call the elements of $\ecal$ \textbf{Euler classes}.
\end{definition}

\begin{ex}
The standard example of an element of
$\ecal$ is given by the dimension function
of a complex representation $V$ of $\torus$ with $V^\torus=0$.
This function sends $H \in \fcal$ to
to the dimension of $V^H$ over $\cc$. We call this element $c^V$.
Note that $c^V c^{V'} = c^{V \oplus V'}$.
\end{ex}

For more details on Euler classes and representations,
see \cite{gre99}[Section 4.6].

\begin{definition}
Define a partial ordering on the functions
$\fcal \to \zz_{\geqslant 0}$ with finite support
by $\nu \geqslant \nu'$
if $\nu(H) \geqslant \nu'(H)$ for each $H \in \fcal$.

For $N$ an $\ocal_\fcal$--module define
$\ecal^{-1} N$ as the colimit of terms
$\Sigma^{\nu} N$ as $\nu$ runs over the
partially ordered set of functions
$\fcal \to \zz_{\geqslant 0}$ with finite support
with $\nu \leqslant \nu + \nu'$
corresponding to
$c^{\nu'} \co \Sigma^{\nu} N \to \Sigma^{(\nu + \nu')} N$.
\[
\ecal^{-1} N
= \colim_{\nu} \Sigma^{\nu} N
\]
\end{definition}
It follows that $c^{\nu} \co \ecal^{-1} N \to \Sigma^{\nu} \ecal^{-1} N$ is an
isomorphism with inverse $c^{-\nu}$. It is easily seen that
$\ecal^{-1} \ocal_\fcal$ is a ring.
To illustrate its structure, we see that
as a vector space, $(\ecal^{-1} \ocal_\fcal)_{2n}$ is
$\prod_{H \in \fcal} \qq $
for $n \leqslant 0$ and is $\oplus_{H \in \fcal} \qq$ for
$n >0$. We also see that there is a natural isomorphism
\[\ecal^{-1} \ocal_\fcal \otimes_{\ocal_\fcal} N
\cong
\ecal^{-1} N.
\]

\begin{definition}
We say that an $\ocal_\fcal$--module $N$
has no \textbf{$\ecal$--torsion} if the map
$N \to \ecal^{-1} N$ is injective.
\end{definition}

We note here that a flat $\ocal_\fcal$-module has no $\ecal$--torsion:
$\ocal_\fcal \to \ecal^{-1} \ocal_\fcal$ is injective, and hence remains
so after tensoring with the flat module.

\begin{definition}\label{def:algmodel}
We define the category $\acal=\acal(\torus)$ as follows.
Its class of objects is the
collection of
$\ocal_\fcal$--module maps
\[
\beta \co N \to \ecal^{-1} \ocal_\fcal \otimes U
\]
with $N$ an $\ocal_\fcal$--module
and $U$ a graded rational vector space,
such that $\ecal^{-1} \beta$ be an isomorphism.
The $\ocal_\fcal$--module $N$ is called the \textbf{nub} and
$U$ is called the \textbf{vertex}.

A map $(\theta, \phi)$ in $\acal$ is a commutative square
as below, where $\theta$ is a map of $\ocal_\fcal$--modules and
$\phi$ is a map of graded rational vector spaces.
\[
\xymatrix@C+1cm{
N \ar[r]^(0.4)\beta \ar[d]_\theta &
\ecal^{-1} \ocal_\fcal \otimes U \ar[d]^{\id \otimes \phi} \\
N' \ar[r]^(0.4){\beta'} &
\ecal^{-1} \ocal_\fcal \otimes U'
}\]
\end{definition}

The relation between this category and rational $\torus$--equivariant
stable homotopy theory is given by the following pair of theorems from \cite{gre99}.
We leave the definition of rational $\torus$--equivariant spectra to the reference.

\begin{theorem}[Greenlees]
The category of rational $\torus$--equivariant spectra up to homotopy
is equivalent to the derived category of $\acal$.
\end{theorem}

For a rational $\torus$--equivariant spectrum $X$, there is an object
$\pi^\acal_*(X)$ of $\acal$. It is constructed in terms of rational equivariant homotopy
groups. We give the definition below, but leave explanations to the reference.
\[
\pi^\acal_*(X) = \Big(
\pi_*^{\torus} (X \smashprod D E \fcal_+) \otimes \qq
\longrightarrow
\ecal^{-1} \ocal_\fcal \otimes (\pi_*(\Phi^{\torus} X) \otimes \qq)  \Big)
\]
There is also an Adams short exact sequence which explains
how to calculate maps in the homotopy category of
rational $\torus$--equivariant spectra.

\begin{theorem}[Greenlees]\label{thm:torusadams}
Let $X$ and $Y$ be $\torus$--equivariant spectra. Then
the sequence below is exact.
\[
0 \to
\ext_\acal (\pi_*^\acal(\Sigma X),\pi_*^\acal( Y))
\to
[X,Y]^{\torus}_* \otimes \qq
\to
\hom_\acal (\pi_*^\acal(X),\pi_*^\acal(Y))
\to 0
\]
\end{theorem}

In \cite{gre99} a model structure is given for the category of objects in
$\acal$ that have a differential. We define what it means to
have a differential and then introduce the model structure. We will leave the proof that
$\acal$ has all small limits and colimits to the next section.

If we think of $\ocal_\fcal$ as an object of $\ch(\qq)$ with trivial
differential, then we can consider the category of
$\ocal_\fcal$--modules in $\ch(\qq)$. Such an object $N$ is an
$\ocal_\fcal$--module in graded vector spaces along with maps
$d_n \co N_n \to N_{n-1}$. These maps
satisfy the relations below.
\[
d_{n-1} \circ d_n =0
\quad
c d_n = d_{n-2} c
\]

\begin{definition}
We define the category $d \acal$, to have objects the
collection of
$\ocal_\fcal$--module maps in in $\ch(\qq)$
\[
\beta \co N \to \ecal^{-1} \ocal_\fcal \otimes U
\]
where
$N$ is a rational chain complex with an action of $\ocal_\fcal$ and
$U$ is a rational chain complex.
Furthermore, we ask that $\ecal^{-1} \beta$ be an isomorphism.

A map $(\theta, \phi)$ in $d \acal$ is then a commutative
square as for $\acal$, such that
$\theta$ is a map in the category of
$\ocal_\fcal$--modules in $\ch(\qq)$
and $\phi$ is a map of $\ch(\qq)$.
\end{definition}

The following result is the subject of \cite[Appendix B]{gre99}.
Note that a map $(\theta, \phi)$ in $d \acal$ is a monomorphism
if and only if both $\theta$ and $\phi$ and injective maps.

\begin{proposition}[Greenlees]\label{prop:injmodelforA(T)}
The category $d \acal$ has a model structure with
cofibrations the monomorphisms and weak equivalences
the quasi--isomorphisms.
This is called the \textbf{injective model structure}.
We write $d \acal_{i}$
to denote this model structure.
Moreover, $\ho(d \acal_{i}) = \acal$.
\end{proposition}

As we shall see shortly, the category $\acal$ has a monoidal product
which induces a monoidal product on $d \acal$.
However the injective model structure does not make
$d \acal$ into a monoidal model category.
This failure occurs because in $d \acal$ the nubs can have $\ecal$--torsion.
This just analogous to how the injective model structure on
$\ch(\zz)$ is not monoidal due to torsion.
This is a serious defect, as we are unable to effectively compare
this monoidal product to the smash product of $\torus$--equivariant spectra.
This defect is further complicated by the lack of projective objects of $\acal$.
Our primary aim is to find a cofibrantly generated monoidal model structure
on $d \acal$ which is Quillen equivalent to the
injective model structure.

We note here that the authors of \cite{BHKKRS}
are working on a sequel investigating whether the
injective model structure on  $d \acal$  (and torsion $\ocal_\fcal$--modules)
are examples of model categories with Postnikov presentations.

\section{Limits}\label{sec:limits}

In this section we give the proof from
\cite{gre99} that
$\acal$ and $d \acal$ have all small limits and colimits
(a necessary condition for model categories).
Along the way we will need to relate $\acal$
to the larger categories $\hat{\acal}$ and $R_\bullet \leftmod$ of Greenlees
\cite{gre12standardalgebra}, we define these new categories below.
For the sake of exposition, we usually only refer to
categories with differentials: $d \acal$, $d \hat{\acal}$
and $d R_\bullet \leftmod$.
Analogues of the results also hold for
the categories without differentials.

The motivation for the definition below is that we want to consider
modules over a \emph{diagram} of (graded) rings. In our case we call the
diagram $R_\bullet$, and its maps are the inclusions:
\[
R_\bullet =(\ocal_\fcal \to \ecal^{-1} \ocal_\fcal \leftarrow \qq).
\]

\begin{definition}
We define $R_\bullet \leftmod$ to be the category of pairs of
morphisms (of $\ecal^{-1} \ocal_\fcal$--modules)
\[
\alpha \co \ecal^{-1} M \to N \quad \quad
\gamma \co \ecal^{-1} \ocal_\fcal \otimes_\qq U \to N
\]
where $M$ is an $\ocal_\fcal$--module,
$N$ is an $\ecal^{-1} \ocal_\fcal$--module and
$U$ is a (graded) $\qq$--module.
We write objects of this category as
quintuples
$(M, \alpha, N, \gamma, U)$.
A morphism of this category is a triple of maps
\[
(f,g,h) \co (M, \alpha, N, \gamma, U) \to (M', \alpha', N', \gamma', U')
\]
where $f \co M \to M'$ is a morphism of $\ocal_\fcal$--modules,
$g \co N \to N'$ is a morphism of $\ecal^{-1} \ocal_\fcal$--modules
and $h \co U \to U'$  is a morphism of $\qq$--modules, such that the following
diagram commutes.
\[
\xymatrix@C+1cm{
\ecal^{-1} M \ar[r]^-{\alpha} \ar[d]^-{\ecal^{-1}f} &
N \ar[d]^-{g} &
\ecal^{-1} \ocal_\fcal \otimes U
\ar[d]^-{1 \otimes h}
\ar[l]_-{\gamma} \\
\ecal^{-1} M' \ar[r]^-{\alpha'} &
N' &
\ecal^{-1} \ocal_\fcal \otimes U'
\ar[l]_-{\gamma'}
}
\]
We may also define $d R_\bullet \leftmod$ just as we defined $d \acal$.
\end{definition}

Observe that limits and colimits in $d R_\bullet \leftmod$
are defined objectwise. In order to define limits in $\acal$
we will construct them in $d R_\bullet \leftmod$
and then use an adjunction to move them to $d \acal$.

\begin{lemma}
There is an adjunction
\[
\inc : d \acal \adjunct
d R_\bullet \leftmod : \Gamma.
\]
where the left adjoint $\inc$ sends an object $(\beta \co N \to \ecal^{-1} \ocal_\fcal \otimes U)$
of $d \acal$ to the quintuple
\[
(N, \beta, \ecal^{-1} \ocal_\fcal \otimes U, \id, \ecal^{-1} \ocal_\fcal \otimes U ).
\]
The functor $\inc$ is full and faithful.
The right adjoint $\Gamma$ is called the \textbf{torsion functor} and is defined below
as the composite of two functors $\Gamma_v$ and $\Gamma_h$,
see Definitions \ref{def:torsionvertical} and \ref{def:torsionhorizontal}.
Moreover, the unit map $A \to \Gamma \inc A$ is an isomorphism
for any object of $d \acal$.
\end{lemma}
\begin{proof}
See \cite[Sections 7 and 8]{gre12standardalgebra} and \cite[Section 20.2]{gre99}.
\end{proof}

We need a category $d \hat{\acal}$ that is half-way between
$d R_\bullet \leftmod$ and $d \acal$. We define $d \hat{\acal}$ as
$d \acal$ without the restriction that $\ecal^{-1} \beta$ be an isomorphism.
Hence $d \acal$ is a full subcategory of $d \hat{\acal}$.
An equivalent definition of $d \hat{\acal}$ is given below as a full subcategory of
$d R_\bullet \leftmod$.
We will generally write an object of $d \hat{\acal}$ as
$(\beta \co N \to \ocal_\fcal \otimes U)$ instead of a quintuple.
We will write $\inc$
for any of these inclusions of subcategories.

\begin{definition}
The category $d \hat{\acal} =d \hat{\acal}(\torus)$ is the full subcategory of
$d R_\bullet \leftmod$ on objects of the form
\[
(N, \beta, \ecal^{-1} \ocal_\fcal \otimes U, \id, U ).
\]
\end{definition}

To define the functor $\Gamma$ it is easiest to describe it as the composite of two functors:
\[
\xymatrix@C+0.5cm{
d \acal
&
\ar[l]^-{\Gamma_h}
d \hat{\acal}
&
\ar[l]^-{\Gamma_v}
\ar@/_2pc/@<-0ex>[ll]_-{\Gamma}
d R_\bullet \leftmod.
}
\]

\begin{definition}\label{def:torsionvertical}
For an object $A= (M,\alpha, N, \gamma, U)$ of $d R_\bullet \leftmod$
We define $\Gamma_v A \in d \hat{\acal}$ to be the
map $\beta$ in the pullback diagram below.
\[
\xymatrix@C+0.5cm{
P \ar[rr]^-{\beta} \ar[d] &&
\ecal^{-1} \ocal_\fcal \otimes U \ar[d]^-{\gamma} \\
M \ar[r] &
\ecal^{-1} M \ar[r]^-{\alpha} &
N
}
\]
\end{definition}
It is easily seen that $\Gamma_v$ is the right adjoint to the inclusion of
$d \hat{\acal}$ into $d R_\bullet \leftmod$.
The functor $\Gamma_h$
is much more complicated to define.
In order to do so, we introduce algebraic spheres and suspensions.

\begin{definition}
If $A=(\beta \co N \to \ecal^{-1} \ocal_\fcal \otimes U)$ is an
element of $d \hat{\acal}$ and $\nu \co \fcal \to \zz_{\geqslant 0}$
is a function with finite support
then we define objects of $d \hat{\acal}$
\[
\begin{array}{r}
\Sigma^\nu A =
\left(  (c^{-\nu} \otimes \id_U) \circ \beta \co
\Sigma^\nu N \to \ecal^{-1} \ocal_\fcal \otimes U \right) \\
\Sigma^{-\nu} A =
\left( (c^{\nu} \otimes \id_U) \circ \beta \co
\Sigma^{-\nu} N \to \ecal^{-1} \ocal_\fcal \otimes U \right)
\end{array}
\]
where
$c^{\nu} \co \ecal^{-1} \ocal_\fcal \overset{\cong}{\longrightarrow} \ecal^{-1} \ocal_\fcal$.
We call the functor $\Sigma^\nu \co d \hat{\acal} \to d \hat{\acal}$
\textbf{suspension by the function $\nu$} and $\Sigma^{-\nu}$
\textbf{desuspension by the function $\nu$}.
It is readily seen that if $A$ is in $d \acal$, then so are
$\Sigma^\nu A$ and $\Sigma^{-\nu} A$.
\end{definition}

\begin{definition}\label{def:algshere}
Define $\ocal_\fcal(\nu)$ to be the submodule of
$\ecal^{-1} \ocal_\fcal$ given by
\[
\ocal_\fcal(\nu)= \{ x \in \ecal^{-1} \ocal_\fcal \ | \ c^\nu x \in \ocal_\fcal \}.
\]
That is $\ocal_\fcal(\nu)$ is generated by the (finite collection) of elements
$e_{\fcal_i} c^{-i}$, where $\fcal_i$ is the set of subgroups
$H$ of $\torus$ where $\nu(H) = i$, for $i \geqslant 0$.

We define $S^\nu \in d \acal$, an \textbf{algebraic sphere},
to be the inclusion map
\[
S^\nu = (\ocal_\fcal(\nu) \to \ecal^{-1} \ocal_\fcal)
\]
equipped with trivial differential.
We can also define $S^{-\nu}\in d \acal$. The nub $\ocal_\fcal(-\nu)$
is the set of those
$x \in \ecal^{-1} \ocal_\fcal$ such that $c^{-\nu} x$ is in $\ocal_\fcal$.
\end{definition}

\begin{ex}
If $V$ is a complex representation of $SO(2)$ with $V^\torus=0$, then there is a $\torus$--equivariant spectrum
$\Sigma^\infty S^V$ and $\pi_*^\acal(\Sigma^\infty S^V) =S^\nu$
where $\nu(H) = \dim_\cc(V^H)$. We call this a \textbf{representation sphere}.
\end{ex}

\begin{lemma}
For $\nu \co \fcal \to \zz_{\geqslant 0}$ of finite support,
multiplication by $c^{-\nu}$ induces an isomorphism in $d \acal$
\[
c^{-\nu} \co \Sigma^\nu S^0 \longrightarrow S^\nu.
\]
\end{lemma}
\begin{proof}
This isomorphism is a more complicated version of the
following simple observation. Let $d$ have degree $-2$.
Define $\qq \langle d^{-n} \rangle$
to be the $\qq[d]$ submodule of $\qq[d,d^{-1}]$
generated by $d^{-n}$. So as a graded $\qq$ module,
$\qq \langle d^{-n} \rangle$
has a copy of $\qq$ in every degree $2k$
for $k \leqslant n$.
Multiplication by $d^{-n}$ is an isomorphism
\[
d^{-n} \co \Sigma^{2n} \qq[d] \to
\qq \langle d^{-n} \rangle.
\]
\end{proof}

For $A$ and $B$ in $d \acal$, we define
$\acal(A, B)_*$ to be the graded set of maps
from the underlying object of $A$ in $\acal$
to the underlying object of $B$ in $\acal$.
We equip this graded $\qq$--module with the
differential induced by the convention
$df_n = d_B f_n + (-1)^{n+1} f_n d_A$.
By considering an object of $\acal$ as an object of
$d \acal$ with no differential, we can extend the definition
of $\acal(A,B)_*$ to allow for the case where $A$
is in $\acal$.

Let $C=(N \to \ecal^{-1} \ocal_\fcal \otimes U)$
be an object of $d \hat{\acal}$. We define a rational chain complex
from $C$ by
\[
\ecal^{-1}N(c^0) = \colim_{\nu} \hat{\acal} (S^{-\nu}, C)_*.
\]
where the colimit runs over the partially ordered set of functions
$\fcal \to \zz_{\geqslant 0}$ of finite support
and uses the inclusion map in $d \hat{\acal}$:  $S^{-\nu'} \to S^{-\nu}$
for $\nu \geqslant \nu'$.
The differential on the graded $\qq$--module
$\hat{\acal} (S^{-\nu}, C)_*$ is
$d f_n = d_C f_n$.

\begin{definition}\label{def:torsionhorizontal}
Let $C=(N \to \ecal^{-1} \ocal_\fcal \otimes U)$
be an object of $d \hat{\acal}$.
We define $\Gamma_h C$ to be the left--hand vertical arrow of the following diagram.
The lower horizontal map is induced by evaluation:
$c^{-\omega} x \otimes (\theta,\phi) \mapsto c^{-\omega} \theta(x)$, where
$(\theta,\phi) \co S^{-\nu} \to C$.
\[
\xymatrix{
N' \ar[r] \ar[d] &
N \ar[d] \\
\ecal^{-1} \ocal_\fcal \otimes \ecal^{-1}N(c^0) \ar[r] &
\ecal^{-1} N
}
\]
\end{definition}

Now that we understand $\Gamma_h$ and $\Gamma_v$,
we may construct limits and colimits
in $d \acal$ and $d \hat{\acal}$.
We leave it as an exercise
to the interested reader to
verify that these are actually constructions of colimits and limits.

\begin{definition}
Let $I$ be some small category and let
$\{ N_i \to \ecal^{-1} \ocal_\fcal \otimes U_i \}$
be the objects of some $I$--shaped diagram in $d \acal$
(or $d \hat{\acal}$).
The colimit over $I$ is
\[
\colim_i N_i \to \ecal^{-1} \ocal_\fcal \otimes (\colim_i U_i) .
\]
The limit in $d \hat{\acal}$ is formed by first including the
objects into $d R_\bullet \leftmod$, taking limits in
this larger category and then applying $\Gamma_v$.
Similarly, the limit in $d \acal$ is formed by first including the
objects into $d R_\bullet \leftmod$, taking limits in
this larger category and then applying $\Gamma$.
\end{definition}

To rephrase the above, the limit of the $I$--shaped diagram
$\{ N_i \to \ecal^{-1} \ocal_\fcal \otimes U_i \}$ in $d \hat{\acal}$
is the map $f$ in the following pullback square.
Equally, the limit of the $I$--shaped diagram
$\{ N_i \to \ecal^{-1} \ocal_\fcal \otimes U_i \}$ in $d \acal$
is $\Gamma_h f$.
\[
\xymatrix{
M \ar[r]^-{f}
\ar[d] &
\ecal^{-1} \ocal_\fcal \otimes \lim(U_i)
\ar[d] \\
\lim(N_i) \ar[r] &
\lim(\ecal^{-1} \ocal_\fcal \otimes U_i)
}
\]

We reiterate that the above constructions and results can be extended
to categories without differentials, since they preserve objects whose differential is zero.

\section{Monoidal products}\label{sec:monoidal}

In this section we give the definitions of the monoidal product and function object
from \cite{gre99}.
We also introduce useful adjoint pairs with the
categories of $\qq$--modules and $\ocal_\fcal$--modules.
Again we use the notation of categories with differentials,
but the obvious analogues hold for categories without differentials.

\begin{definition}
For
$\beta \co N \to \ecal^{-1} \ocal_\fcal \otimes U$ and
$\beta' \co N' \to \ecal^{-1} \ocal_\fcal \otimes U'$ in $d \acal$
(or $d \hat{ \acal}$), their
\textbf{tensor product} is
\[
\beta \otimes \beta' \co
N \otimes_{\ocal_\fcal} N'
\to
(\ecal^{-1} \ocal_\fcal \otimes U)
\otimes_{\ocal_\fcal}
(\ecal^{-1} \ocal_\fcal \otimes U')
\cong
\ecal^{-1} \ocal_\fcal \otimes (U \otimes_\qq U')
\]
The unit of this monoidal product is the object
$S^0 = (i \co \ocal_\fcal \to \ecal^{-1} \ocal_\fcal \otimes \qq$).
The differential is given by the usual rule:
$d_{n,m} = d_n \otimes 1 + (-1)^n 1 \otimes d_m$.

Similarly the tensor product in $d R_\bullet \leftmod$
is defined objectwise and its unit is
$(\ocal_\fcal, i, \ecal^{-1} \ocal_\fcal, i', \qq)$
where $i$ and $i'$ are the inclusions.
\end{definition}

\begin{ex}
The tensor product of two algebraic
spheres $S^\nu$ and $S^{\nu'}$ is the algebraic sphere $S^{\nu + \nu'}$.
\end{ex}

\begin{lemma}
The functor $\inc$ is a symmetric  monoidal functor from
$d\acal$ to $R_\bullet \leftmod$.
\end{lemma}

The monoidal product on $\acal$ is homotopically meaningful, as we can see from following result, which is \cite[Theorem 24.1.2]{gre99}.

\begin{theorem}[Greenlees]\label{thm:greenleesmonoidal}
For rational $\torus$-spectra $X$ and $Y$
there is a short exact sequence in $\acal$
\[
0
\to
\pi_*^\acal(X) \otimes \pi_*^\acal(Y)
\to \pi_*^\acal(X \smashprod Y)
\to
\Sigma \tor (\pi_*^\acal(X),\pi_*^\acal(Y))
\to 0.
\]
\end{theorem}

The monoidal structure on each of these categories is closed,
that is, each category has an internal function object.
We deal with $R_\bullet \leftmod$ first, as the other two
function objects are built from this.

\begin{definition}\label{def:intfunction1}
In $R_\bullet \leftmod$ the internal function object is defined as
\[
F_{R^\bullet} \left( (M,\alpha, N, \gamma, U) ,  (M',\alpha', N', \gamma', U') \right)
=
(D,\theta, \hom_{\ecal^{-1} \ocal_\fcal}(N,N'), \phi, E)
 \]
with $D$, $E$, $\theta$ and $\phi$ defined in the pullback diagrams below.
In the following, we let $i^*$ denote the context-appropriate forgetful functor and
let
$\bar{\alpha} \co M \to i^*N$ be the adjoint to $\alpha$
and $\bar{\gamma} \co U \to i^*N$. be the adjoint to $\gamma$.
\[
\xymatrix{
D \ar[r]^-{\bar{\theta}} \ar[dd] & i^* \hom_{\ecal^{-1} \ocal_\fcal}(N,N') \ar[d]
&
i^* \hom_{\ecal^{-1} \ocal_\fcal}(N,N') \ar[d] & E \ar[l]_-{\bar{\phi}} \ar[dd] \\
& \hom_{\ocal_\fcal}(i^*N,i^*N') \ar[d]^-{\bar{\alpha}^*}
&
\hom_\qq(i^*N,i^*N') \ar[d]_-{\bar{\gamma}^*} \\
\hom_{\ocal_\fcal}(M,M') \ar[r]^-{\bar{\alpha}'_*} &
\hom_{\ocal_\fcal}(M,i^*N')
&
\hom_\qq(U,i^*N') &
\hom_\qq(U,U') \ar[l]_-{\bar{\gamma}'_*}
}
\]
\end{definition}

\begin{definition}\label{def:intfunction2}
Letting $\inc$ denote the inclusion of either of
$d \hat{\acal}$ or $d \acal$ into $R_\bullet \leftmod$,
we define the internal function object of $d \hat{\acal}$ to be:
\[
F_{\hat{\acal}} \left( A ,  B \right)
= \Gamma_v F_{R^\bullet} ( \inc A, \inc B ) .
\]
Similarly, we define the internal function object of $d \acal$ to be:
\[
F_{\acal} \left( A ,  B \right)
= \Gamma F_{R^\bullet} ( \inc A, \inc B ) .
\]
\end{definition}

\begin{lemma}
The categories $d R_\bullet \leftmod$, $d \hat{\acal}$ and $d\acal$ are all closed monoidal categories.
\end{lemma}
\begin{proof}
We leave the first two cases as an exercise and concentrate on the third.
Since $\inc$ is full, faithful and monoidal,
\[
d\acal(A \otimes B, C) \cong
d R_\bullet \leftmod (\inc A \otimes \inc B, \inc C).
\]
Hence $d\acal(A \otimes B, C)$ is naturally isomorphic to
$d \acal (A , \Gamma F_{R_\bullet}(\inc B, \inc C))$.
\end{proof}

One reason for the complicated form of the above definition is that
we need to make sure that our structure maps have the correct form.
For example, in Definition \ref{def:intfunction1} $E$ must be a $\qq$--module and $\phi$ be of the form
\[
\phi \co \ecal^{-1} \ocal_\fcal \otimes E \to  \hom_{\ecal^{-1} \ocal_\fcal}(N,N').
\]
Using the extra restrictions on objects of $d \hat{\acal}$ and $d \acal$
we can give a more direct construction of $F_{\acal}(-,-)$
and $F_{\hat{\acal}}(-,-)$. In particular, the pullback squares below are
essentially the `adjoints' of those above. We leave it to the reader to verify that
the constructions in the following example agree with the definitions above.

\begin{ex}
Consider two elements of $d\hat{\acal}$,
\[
A= (\beta \co N \to \ecal^{-1} \ocal_\fcal \otimes U)
\quad  {and} \quad
B= (\beta' \co N' \to \ecal^{-1} \ocal_\fcal \otimes U').
\]
The
\textbf{function object} $F_{\hat{\acal}}(A,B) \in d \hat{\acal}$ is the map
$\delta$,  as defined by the pullback square below.
\[
\xymatrix{
Q
\ar[dd]
\ar[r]^(0.4)\delta &
\ecal^{-1} \ocal_\fcal \otimes \hom_{\qq} (U,U')
\ar[d] \\
&
\hom_{\ocal_\fcal} (
\ecal^{-1} \ocal_\fcal \otimes U,
\ecal^{-1} \ocal_\fcal \otimes  U')
\ar[d]^{\beta^*} \\
\hom_{\ocal_\fcal} (N,N')
\ar[r]_-{\beta_*'} &
\hom_{\ocal_\fcal} (N,
\ecal^{-1} \ocal_\fcal \otimes  U')
}
\]
Now assume that $A$ and $B$ are in $d \acal$. Then we may construct the map
$\delta$ as above and we see that $F_{\acal}(A,B) \in d \acal$
is $\Gamma_h \delta$.
\end{ex}

For the sake of notation we often just write $F$ for any of these internal function objects. We can use the monoidal product and internal
function object to show that $d\acal$ is enriched, tensored and cotensored
over $\ch(\qq)$.

\begin{definition}\label{def:modulefunctors}
For $K \in \ch(\qq)$ we define $LK \in d \acal$ as
\[
LK= (i \otimes \id_K \co
\ocal_\fcal \otimes K \to
\ecal^{-1} \ocal_\fcal \otimes K)
\]

For $A$ and $B$ in $d \acal$, we define
$\acal(A, B)_*$ to be the graded set of maps
of $\acal$ (ignoring the differential). We then equip this
graded $\qq$--module with the differential induced by the convention
$df_n = d_B f_n + (-1)^{n+1} f_n d_A$.
This construction gives a functor as below.
\[
R \co d \acal \to \ch(\qq) \quad \quad
RA := \acal(S^0, A)_*
\]
\end{definition}

The functors $L$ and $R$ form an adjoint pair between
$\ch(\qq)$ and $d \acal$. Furthermore, they
give $d \acal$
the structure of a closed
$\ch(\qq)$--module in the sense of Hovey \cite[Section 4.1]{hov99}.

This module structure and the closed monoidal product interact to give
$d \acal$  a tensor product, a cotensor product and
an enrichment over $\ch(\qq)$.
Let $K \in \ch(\qq)$
and $A = (\beta \co N \to \ecal^{-1} \ocal_\fcal \otimes U)$ in
$d \acal$.
Their \textbf{tensor product} $A \otimes K$ is defined to be  $A \otimes LK$.
Thus $A \otimes K$ is given by
\[
\beta \otimes \id_K  \co N \otimes_\qq K \to
\ecal^{-1} \ocal_\fcal \otimes (U \otimes_\qq K).
\]
The \textbf{cotensor product} $A^K$ is defined as
$F(LK, A)$.
The \textbf{enrichment} is given by  $R F(A,B)$
for $A$ and $B$ in $d \acal$.
The enrichment, tensor and cotensor are related by the natural
isomorphisms below.
\[
d \acal(A, B^K)
\cong
d \acal(A \otimes K, B)
\cong
\ch(\qq)(K, R F(A,B))
\]

We also need to relate $d \acal$ to the category of
$d \ocal_\fcal$--modules. In particular, the following construction will be
essential to Proposition \ref{prop:dualfingenproj}.

\begin{definition}
There is an adjunction
\[
g_* : d \acal
 \adjunct
d \ocal_\fcal \leftmod  : g^*
\]
The left adjoint $g_*$ sends an
object of $d\acal$ to its nub.
For $N$ in $d \ocal_\fcal \leftmod$
we define the right adjoint by
$g^* N = \Gamma (N,0,0,0,0)= \Gamma_h (N \to 0)$,
where $(N,0,0,0,0) \in d R_\bullet \leftmod$
and $(N \to 0) \in d \hat{\acal}$.
\end{definition}

More specifically, we may also
describe the object $g^* N \in d \acal$ as the left hand vertical in the
pullback diagram below, where $\ecal^{-1} N$ is considered an
object of $\ch(\qq)$.
\[
\xymatrix{
P \ar[r] \ar[d] &
N \ar[d] \\
\ecal^{-1} \ocal_\fcal \otimes_\qq \ecal^{-1}N \ar[r] &
\ecal^{-1} N
}
\]

\begin{lemma}
The functor $g^*$ from the category of $\ocal_\fcal$--modules to $\acal$
is exact and commutes with filtered colimits.
\end{lemma}
\begin{proof}
The functor sending an $\ocal_\fcal$--module $N$ to $N \to 0$
in $\hat{\acal}$ is clearly exact.
By \cite[Proposition 20.3.4]{gre99}, we see that the
right derived functors of $\Gamma_h$ are all zero.
Hence $\Gamma_h$ is exact and $g^*$ is also exact.
Since $g^*$ is defined in terms of a forgetful functors,
tensor products and a pullback,
it must commute with filtered colimits.
\end{proof}

Notice that the evaluation functor
which sends an object of $d \acal$ to its vertex is also a \emph{left} adjoint.
The right adjoint to evaluation at the vertex
is the functor which sends $V \in \ch(\qq)$ to the object
$\id \co \ecal^{-1} \ocal_\fcal \otimes V \to \ecal^{-1} \ocal_\fcal \otimes V$.

Another curious feature about the category $d\acal$ is that
the nub of every object is a \textbf{Hausdorff module}, that is
the natural map of $\ocal_\fcal$-modules
\[
M \to \prod_{H} e_H M
\]
is injective, see \cite[Section 5.10]{gre99}.
Hence the cokernel of the natural map
$\oplus_{H} \qq[c_H] \to \ocal_\fcal$ cannot occur as the nub of
an element of $\acal$.

In general it is hard to construct left adjoints in to $d \acal$, as
the nubs must be Hausdorff and the structure map must be an isomorphism after
inverting $\beta$. Conversely, right adjoints into $d \acal$ are easier to construct as
we can land in the simpler categories $d \hat{\acal}$ and $d R_\bullet \leftmod$
and then apply $\Gamma_h$ or $\Gamma$. The reader should compare this behaviour with $d R_\bullet \leftmod$,
where the evaluation functors are all \emph{right} adjoints. That is, the functor which sends
$X=(M, \alpha, N, \gamma, U)$ to $M \in d \ocal_\fcal \leftmod$ is a right adjoint, as is
the functor which sends $X$ to $N \in d \ecal^{-1} \ocal_\fcal \leftmod$ or the functor which sends
$X$ to $U \in \ch(\qq)$.

\section{Dualisable objects of \texorpdfstring{$\acal(\torus)$}{A(T)}}\label{sec:duals}

In this section we introduce the class of dualisable objects of our category $\acal$
and characterise them as objects whose nub is finitely generated and projective, Proposition \ref{prop:dualfingenproj}.
We then use this characterisation to show that there is only a set of isomorphism
classes of dualisable objects, Corollary \ref{cor:dualisablesformaset}.
We construct an important collection of dualisable
objects called the wide spheres. We will use the
dualisable objects and wide spheres in the next section to
construct the desired monoidal model structure on $d \acal$.
The results of this section are stated in terms of $\acal$. They can
all be extended to categories with differential.

\begin{definition}
A object $A$ of $\acal$ is said to be \textbf{dualisable} if
for any $B \in \acal$ the canonical map
$F(A,S^0) \otimes B \to F(A,B)$ is an isomorphism.
The \textbf{dual} of an object $B$ is the object $DB := F(B,S^0)$.
\end{definition}

We may also define dualisable objects in $\ocal_\fcal$--modules and
(graded) $\qq$--modules. Recall that a graded $\qq$--module is
dualisable if and only if it is finite dimensional.
Equally an $\ocal_\fcal$--module is dualisable if it is
finitely generated and projective (such objects are retracts of
finite products of $\ocal_\fcal$).
Dualisable objects satisfy a number
of useful properties, we state some below for $\acal$, but the obvious
analogues hold for $\ocal_\fcal$--modules and (graded) $\qq$--modules.

\begin{lemma}
Let $A$ be a dualisable object of $\acal$. Then $DA$ is dualisable, $D(DA) \cong A$ and
$A$ is flat (that is, $- \otimes A$ is exact). For any $B$ and $C$ in $\acal$
we have a natural isomorphism
\[
F(B,A \otimes C) \cong F(B \otimes DA, C)
\]
\end{lemma}
\begin{proof}
Most of these statements are proven in Lewis, May and Steinburger
\cite[Section III.1]{lms86}.
To see that dualisable implies flat, we must prove that
$A \otimes -$ is an exact functor.
It is always right exact and it is isomorphic to
$F(DA, -)$, which is always left exact.
\end{proof}

Following Hovey \cite{hov04}, we make the following definition.
Other sources call modules satisfying this
condition small or compact.
\begin{definition}
We say that an object $X$ of a category $\ccal$ is
\textbf{finitely presented} if the
functor $\ccal(X,-)$ commutes with filtered colimits.
\end{definition}
It is a standard result that a module over a ring $R$ is finitely
presented if and only if it is the cokernel of a
map of free $R$-modules of finite rank.
In particular a $\qq$--module is finitely presented
if and only if it has finite dimension.

\begin{ex}
The algebraic sphere $S^\nu$
is finitely presented. In particular,
$S^0 = (\ocal_\fcal \to \ecal^{-1} \ocal_\fcal)$
is finitely presented.
\end{ex}

The ring $\ecal^{-1} \ocal_\fcal$ is the filtered colimit over functions
$\nu \co \fcal \to \zz_{\geqslant 0}$ with finite support of
$\Sigma^{\nu} \ocal_\fcal$.
Indeed, for any $\ocal_\fcal$--module $M$,
$\ecal^{-1} M$ is the filtered colimit over $\nu$ of
$\Sigma^{\nu} M$.
Hence, if $N$ is a finitely presented  $\ocal_\fcal$--module, then
\[
\hom_{\ocal_\fcal}(N, \ecal^{-1} M)
\cong \colim_\nu \Sigma^{\nu} \hom_{\ocal_\fcal}(N, M)
\cong \ecal^{-1} \hom_{\ocal_\fcal}(N, M)
\]
Recall that a finitely generated projective module
is finitely presented.
These two facts allow us to prove the following analogue
of \cite[Propositions 1.3.2, 1.3.3 and 1.3.4]{hov04}.
The first states that if $A \in \acal$ is nice, then
$F(A,-)$ is much simpler to describe.

\begin{proposition}\label{prop:specialfunction}
Let $A=(\beta \co N \to \ecal^{-1} \ocal_\fcal \otimes U)$
and $B=(\beta' \co N' \to \ecal^{-1} \ocal_\fcal \otimes U')$
be objects of $\acal$ and assume that the nub of
$A$ is finitely presented and has no $\ecal$--torsion.
Then $F(A,B)$ is isomorphic to
\[
\hom_{\ocal_\fcal}(N, N') \longrightarrow
\ecal^{-1} \hom_{\ocal_\fcal}(N, N') \cong
\ecal^{-1} \ocal_\fcal \otimes \hom_\qq(U,U')
\]
\end{proposition}
\begin{proof}
Since $N$ is finitely presented, there is a surjection
from some finite sum of copies of $\ocal_\fcal$ to $N$.
Hence we have a surjection from some finite sum of copies of
$\ecal^{-1} \ocal_\fcal$ to $\ecal^{-1} N$
and thus a surjection to $\ecal^{-1} \ocal_\fcal \otimes U$.
It follows that $U$ must be finite dimensional.
Hence the diagonal map
\[
\ecal^{-1} \ocal_\fcal \otimes \hom_\qq(U,U') \longrightarrow
\hom_{\ocal_\fcal}(\ecal^{-1} \ocal_\fcal \otimes U, \ecal^{-1} \ocal_\fcal \otimes U')
\]
is an isomorphism. The target of this map is isomorphic (using $\beta$ and $\beta'$) to
the domain of the map below, which is induced by
$\alpha \co N \to \ecal^{-1} N$.
\[
\alpha^* \co \hom_{\ocal_\fcal}(\ecal^{-1} N, \ecal^{-1} N')
\longrightarrow
\hom_{\ocal_\fcal}(N, \ecal^{-1} N')
\]
Since $N$ has no $\ecal$--torsion,
the above map is in fact an isomorphism.
The module $N$ is also finitely presented, so we see that the codomain of the above is naturally isomorphic to
$\ecal^{-1} \hom_{\ocal_\fcal}(N, N')$.
Hence we have specified an isomorphism
\[
\ecal^{-1} \ocal_\fcal \otimes \hom_\qq(U,U')
\longrightarrow
\ecal^{-1} \hom_{\ocal_\fcal}(N, N').
\]
Recall that the definition of $F(A,B)$ is given in terms of a pullback
over the diagonal map and the map $\alpha^*$. It follows that
$F(A,B)$ is given by applying the torsion functor to the map
\[
\hom_{\ocal_\fcal}(N, N') \longrightarrow
\hom_{\ocal_\fcal}(N, \ecal^{-1} N') \longrightarrow
\ecal^{-1} \hom_{\ocal_\fcal}(N, N').
\]
The codomain of the above map is  isomorphic to
$\ecal^{-1} \ocal_\fcal \otimes \hom_\qq(U,U')$,
hence the torsion functor has no effect and the result is proven.
\end{proof}

The next result relates the notion of being finitely presented in the category of $\ocal_\fcal$--modules to the notion of being finitely presented in $\acal$.

\begin{proposition}
If $F(A,-)$ preserves filtered colimits
in $\acal$, then $A$ is finitely presented.
If $A \in \acal$ is finitely presented in $\acal$,
then its nub is finitely presented  in
the category of $\ocal_\fcal$--modules.
Conversely, if the nub of $A$ is finitely presented in
the category of $\ocal_\fcal$--modules and has no $\ecal$--torsion,
then $A$ is finitely presented in $\acal$.
\end{proposition}
\begin{proof}
To prove the first statement we take some
filtered colimit diagram $B_i$ in $\acal$.
Then we have isomorphisms as below.
\[
\begin{array}{ccccl}
\colim_i \acal(A, B_i)
& \cong &
\colim_i \acal(S^0, F(A,B_i))
& \cong &
\acal(S^0, \colim_i F(A,B_i)) \\
& \cong &
\acal(S^0, F(A,\colim_i B_i))
& \cong &
\acal(A,\colim_i B_i)
\end{array}
\]
The second statement follows from the
fact that the functor $g^*$
from $\ocal_\fcal \leftmod$ to $\acal$
commutes with filtered colimits.

For the converse, let
$A=(\beta \co N \to \ecal^{-1} \ocal_\fcal \otimes U)$
with $N$ finitely presented as an $\ocal_\fcal$--module
and with no $\ecal$--torsion. As with the proof of the previous result, we see that $U$
is finite--dimensional.
Let $B_i$ be a filtered system of elements of $\acal$
with nubs $N_i$ and vertices $U_i$.
We must prove that the canonical map
\[
\colim_i  F(A, B_i)
\longrightarrow
F(A, \colim_i B_i)
\]
is an isomorphism in $\acal$.
This follows immediately from the description of $F(A, B_i)$ in
Proposition \ref{prop:specialfunction}
and the facts below.
\[
\begin{array}{rcl}
\colim_i \hom_{\ocal_\fcal}(N,  N_i)
& \cong &
\hom_{\ocal_\fcal}(N, \colim_i N_i) \\
\colim_i \hom_{\qq}(U,  U_i)
& \cong &
\hom_{\qq}(U, \colim_i U_i)
\end{array}
\]
\end{proof}

Using the previous two propositions,
we may give a characterisation of the dualisable
objects of $\acal$. Note that by the proof of
Proposition \ref{prop:specialfunction}.
The vertex of an object of $\acal$ whose nub
is finitely generated is finite--dimensional.

\begin{proposition}\label{prop:dualfingenproj}
An object $A$ of $\acal$ is dualisable if and only
if its nub is finitely generated and projective as an
$\ocal_\fcal$--module.
\end{proposition}
\begin{proof}{
Let $A=(\beta \co N \to \ecal^{-1} \ocal_\fcal \otimes U)$
and $B=(\beta' \co N' \to \ecal^{-1} \ocal_\fcal \otimes U')$.
Assume that $N$ is finitely generated and projective as
an $\ocal_\fcal$--module. Then $N$ is finitely presented and has no $\ecal$--torsion, hence
\[
F(A,B) = (\hom_{\ocal_\fcal}(N, N') \longrightarrow
\ecal^{-1} \ocal_\fcal \otimes \hom_\qq(U,U')
)
\]
Since $N$ and $U$ are finitely generated and projective, it follows that
they are dualisable. Hence
$F(A,B)$ is isomorphic to the map below.
But that is simply $DA \otimes B$, so $A$ is dualisable.
\[
\hom_{\ocal_\fcal}(N, \ocal_\fcal) \otimes_{\ocal_\fcal} N'
\longrightarrow
\ecal^{-1} \ocal_\fcal \otimes \hom_\qq(U,\qq) \otimes U'
\]

For the converse, assume that $A$ is dualisable.
The functor $F(A,-)$ commutes with colimits as it is isomorphic
to $DA \otimes -$.
Hence the nub of $A$ is finitely presented. Furthermore, $A$ is flat,
so the nub of $A$ cannot have any $\ecal$--torsion.
We must now prove that the nub of $A$ is projective.

Let $E$ be an exact sequence in $\ocal_\fcal$--modules.
Recall the functor $g^* \co \ocal_\fcal \leftmod \to \acal$
from the previous section.
We have shown that this functor is exact, so $g^* E$ is an
exact sequence in $\acal$.
The sequence $F(A,g^* E)$ is isomorphic to
$DA \otimes g^* E$, hence both these sequences are exact.
We also see that $F(A,g^* E)$ is isomorphic to
$g^* \hom_{\ocal_\fcal}(N,E)$ since the left adjoint
to $g^*$ is monoidal.
Now we must show that $\hom_{\ocal_\fcal}(N,E)$ is an exact
sequence. This amounts to proving that if
$\alpha \co X \to Y$ is a map of $\ocal_\fcal$--modules,
such that $g^* \alpha$ is a surjection, then
$\alpha$ itself is a surjection.
But this follows by looking at the pullback
diagrams defining $g^* \alpha$ and noting that the
map from the nub of $g^*Y$ to $Y$ is a surjection.
Thus we conclude that the sequence $\hom_{\ocal_\fcal}(N,E)$ is exact and
hence $N$ is projective.
}
\end{proof}

Note further that the structure map of any dualisable
object is injective as the nub has no $\ecal$--torsion.

\begin{corollary}\label{cor:dualisablesformaset}
The collection of isomorphism classes of
dualisable objects is a set.
\end{corollary}
\begin{proof}
Consider some dualisable object
\[
\beta \co N \to \ecal^{-1} \ocal_\fcal \otimes U
\]
of $\acal$.
We know that $N$ has no $\ecal$-torsion, so $\beta$
is a monomorphism. It is simple to check that this
object is isomorphic to an inclusion of $\ocal_\fcal$-modules:
\[
\beta' \co N' \to \ecal^{-1} \ocal_\fcal \otimes U
\]
For fixed $U$, there is only a set of such inclusions. Hence,
up to isomorphism in $\acal$, there is only a set of objects of $\acal$
with vertex $U$.

Now note that if $\phi \co U \to U'$ is an isomorphism of
graded $\qq$-modules, then $\beta'$ and $(\id \otimes \phi) \circ  \beta'$
are isomorphic objects of $\acal$.
The collection of isomorphism classes of graded $\qq$-modules forms a set.
Hence the collection of isomorphism classes of objects of $\acal$ forms a set.
\end{proof}

\begin{definition}\label{def:setofreps}
We let $\pcal$ denote a set of representatives for the
isomorphisms classes of dualisable objects.
\end{definition}

We need to introduce a special and useful collection of dualisable objects of
$\acal$, which are only slightly more complicated than the
algebraic spheres.
We will use this class in the proof of Theorem \ref{thm:modelstructure}.

\begin{definition}\label{def:widesphere}
A \textbf{wide sphere} is an object
$S \to \ecal^{-1} \ocal_\fcal  \otimes T$
with $T$ a finitely generated vector space on elements $t_1, \dots, t_d$.
The nub $S$ is the $\ocal_\fcal$ submodule of
$\ecal^{-1} \ocal_\fcal  \otimes U$
generated by elements
$c^{a_1} \otimes t_1, \dots, c^{a_d} \otimes t_d$
for Euler classes $c^{a_i}$ and
an element $\Sigma_{i=1}^d \sigma_i \otimes t_i$ of
$\ecal^{-1} \ocal_\fcal  \otimes T$.
\end{definition}

The reason we study wide spheres is that they are flat and
there are enough wide spheres in the sense that
any object in $\acal$ admits a surjection
from a coproduct of wide spheres, see
\cite[Lemma 22.3.4]{gre99}.
It follows that they can be used to define a derived monoidal product.
We reproduce the proof that there are enough wide spheres.

Take another object $A= (\beta \co N \to \ecal^{-1} \ocal_\fcal \otimes U)$.
We want to show that for any $n \in N$ or any $u \in U$ there is
a wide sphere and a map to $A$ such that $n$ or $u$
is in the image of this map. Since $\ecal^{-1} \beta$ is
an isomorphism, it suffices to only consider elements of the nub.

So consider $n \in N$ with
$\beta(n) = \Sigma_{i=1}^d \sigma_i \otimes u_i$.
For each $i$ there is an element $p_i \in N$ with
$\beta(p_i) = c^{b_i} \otimes u_i$ for $b_i$ function
with finite support. We may assume that we have chosen the
$b_i$ so that
$\sigma_i c^{b_1 + \dots +b_d}/c^{b_i} \in \ocal_\fcal$.
We can always multiply the $c^{b_i}$ by some Euler class
so that this holds. Now we must find another Euler class.
We know that
\[
\beta \left(
\Sigma_{i=1}^d
\sigma_i c^{b_1 + \dots +b_d}c^{-b_i} \cdot p_i
\right)
=
\Sigma_{i=1}^d
\sigma_i c^{b_1 + \dots +b_d} \otimes u_i
=
\beta \left( c^{b_1 + \dots +b_d} \cdot n \right)
\]
Since $\ecal^{-1} \beta$ is an isomorphism,
there must be some Euler class $c^b$ such that
\[
\Sigma_{i=1}^d
c^b \sigma_i c^{b_1 + \dots + b_d}c^{-b_i} \cdot p_i
=
c^b c^{b_1 + \dots + b_d} \cdot n.
\]
Now we can define our wide sphere.
Let $T$ be the subspace of $U$ generated by
the elements $u_i$.
Let $S$ be the $\ocal_\fcal$--submodule of
$\ecal^{-1} \ocal_\fcal \otimes V$ generated
by the elements $\Sigma_{i=1}^d \sigma_i \otimes u_i$
and $c^{b+b_1+ \dots +b_d} \otimes u_i$.
The structure map is the inclusion and it is clearly
an isomorphism after inverting $\ecal$.

We are ready to describe our desired map from this wide sphere to $A$.
On the nub it sends
$\Sigma_{i=1}^d \sigma_i \otimes u_i$ to $n$ and
$c^{b+b_1+ \dots +b_d} \otimes u_i$ to
$c^{b+b_1+ \dots +b_d}c^{-b_i} \cdot p_i$.
On the vertex it is the inclusion.
It is a useful exercise to check that
this defines a map in $\acal$ from the wide sphere to $A$.
The Euler classes $c^{b_i}$ and $c^b$ are needed to ensure that
the non--trivial relation between
$\Sigma_{i=1}^d \sigma_i \otimes u_i$ and the terms
$c^{b+b_1+ \dots +b_d} \otimes u_i$ in the nub of the wide sphere
is replicated by their images in $N$.

\begin{lemma}\label{lem:wideduals}
Any wide sphere is dualisable.
\end{lemma}
\begin{proof}
The nub of a wide sphere is finitely generated by definition.
The nubs are projective by
\cite[Lemma 23.3.3]{gre99}, hence
the wide spheres are dualisable by
Proposition \ref{prop:dualfingenproj}.
\end{proof}

Note that any algebraic sphere $S^\nu$ is in particular
a wide sphere (and is dualisable).

\section{The dualisable model structure}\label{sec:dualisable}

In this section we complete our construction of a monoidal model structure on
$d \acal$, the main results are Theorems \ref{thm:modelstructure} and \ref{thm:monoidal}.
The inspiration for the method used comes from Barnes and Roitzheim \cite[Remark 6.11]{brfranke}.
The key facts are that we need enough dualisable objects (there
is a surjection from a coproduct of wide spheres to any object of $\acal$)
and that the collection of isomorphism classes of dualisable objects
forms a set. Both of these statements have been proven in the previous
section, so we are ready to construct our model structure.
Our starting point is to prove a general result on
the existence of model structures on $d \acal$ whose
weak equivalences are the homology isomorphisms.

\begin{proposition}\label{prop:smithmodel}
Let $I$ be a set of monomorphisms such that
the maps with the right lifting property with respect to $I$ are
homology isomorphisms. Then there is a cofibrantly generated
model structure
on $d \acal$ with weak equivalences
the homology isomorphisms
and $I$ as the set of generating cofibrations.
\end{proposition}
\begin{proof}
We wish to use Smith's theorem, which appears as Beke \cite[Theorem 1.7]{beke00},
to construct model structures on $d \acal$.
That theorem uses some technical set--theoretic terms that we need to mention, but do not want to define.
They are the \textbf{solution set condition} as introduced in \cite[Definition 1.5]{beke00}
and the notions of
\textbf{locally presentable categories} and \textbf{locally accessible categories} as defined in
Borceux \cite[Sections 5.2 and 5.3]{bor94}.
Given our assumptions, we must prove that $d \acal$ is a locally presentable category,
and that the set of homology isomorphisms
of $d \acal$ satisfies the solution set condition.

By \cite[Proposition 3.10]{beke00} we see that
$\acal$ is locally presentable
if and only if it has a set of objects
$G_i$ such that $\acal(\oplus_i G_i,-)$ is a faithful functor from
$\acal$ to sets. Such a set exists by \cite[Lemma 22.3.4]{gre99}
which says that there are enough wide spheres
(see definition \ref{def:widesphere}).
We can extend this result to $d \acal$ by taking the set to be
the collection of wide spheres tensored with $D^1 \in \ch(\qq)$.
Hence $d \acal$ is locally presentable.

We need to know that the set of homology isomorphisms
of $d \acal$ satisfies the solution set condition.
By \cite[Propositions 1.15 and 3.13]{beke00}
we must prove that the homology isomorphisms are an
accessible category.
This follows from the following facts: the homology functor
$\h_* \co d \acal \to \acal$ commutes with
filtered colimits, the isomorphisms of $\acal$
are accessible (they are so in any locally presentable category)
and \cite[Proposition 1.18]{beke00}.
\end{proof}

Now we are ready to use the dualisable objects to make
our desired monoidal model structure on $\acal(\torus)$.
As is standard, we write $S^{n-1}$ for that object of $\ch(\qq)$ which is
$\qq$ concentrated in degree $n-1$ and $D^n$ for the chain complex
with $\qq$ in degrees $n$ and $n-1$ (with the identity as the differential).
We let $i_n$ denote the inclusion map from $S^{n-1}$ to $D^n$.
Recall $\pcal$ from Definition \ref{def:setofreps}, a set or representatives
of the isomorphism classes of dualisable objects.

\begin{theorem}\label{thm:modelstructure}
There is a cofibrantly generated model structure
on $d \acal$ with weak equivalences the
homology isomorphisms. The generating cofibrations
have the form
\[
i_n \otimes P \co S^{n-1} \otimes P \longrightarrow D^n \otimes P
\]
for $P \in \pcal$ and $n \in \zz$.
We call this model structure the \textbf{dualisable model structure}
and denote it $d \acal_{dual}$.
\end{theorem}
\begin{proof}
We have a set of generating cofibrations $I$ and
we must show that any map $f \co A \to B$
with the right lifting property
with respect to $I$ is a homology isomorphism.
We see that such a map must have the property that for any
dualisable object $P$, the induced map of chain complexes
\[
f_* \co \acal(P, X)_* \to \acal(P,Y)_*
\]
is a homology isomorphism.
In particular $f_* $ has the right lifting property with respect to
$0 \to D^n \otimes \qq$ for $n \in \zz$. In turn, $f$ has the
right lifting property with respect to any map of the form
\[
0 \to D^n \otimes P
\]
for $P \in \pcal$ and $n \in \zz$. By the discussion after Definition
\ref{def:widesphere}, it follows that $f$ must be surjective.
It follows that the homotopy fibre $Z$ of $f=(\theta,\phi)$ is given by
\[
(\ker \theta \to \ecal^{-1} \ocal_\fcal \otimes \ker \phi ) \in d \acal.
\]
The chain complex $\acal(P,Z)_*$ is precisely the homotopy fibre of
$f_*$ and hence is acylic.
We must show that this implies that $Z$ has trivial homology.
It suffices to show that any cycle $n$ in the nub of
$Z$ is also a boundary. By the discussion after Definition
\ref{def:widesphere} there is a map $\alpha$ from a wide sphere
$P$ to $Z$ with this cycle $n$ in its image.
We must also ensure that
$\alpha$ is a cycle in
$\acal(P,Z)_*$.
To do so, we follow the method of the previous section until
we come to choose the elements $p_i$. For each $i$, $d p_i$
must be $\ecal$--torsion, as each $u_i$ is a cycle.
So there is some Euler class $c^{r_i}$
with $c^{r_i} p_i$ a cycle. We can then use these elements
instead of the $p_i$ in the construction of the wide sphere.
With these choices, $\alpha$ is a cycle.
Since $\acal(P,Z)_*$ has no homology, $\alpha$ is a boundary,
hence so is $n$.
\end{proof}

There is also a relative projective model structure on $d \acal$, which has the
same cofibrations as the dualisable model structure,
but has generating acyclic cofibrations given by
$0 \to D^n P$ for $n \in \zz$ and $P \in \pcal$, see Definition \ref{def:setofreps}.
unfortunately the weak equivalences are not the homology isomorphisms,
so we will not use this model structure.
We note that one alternative approach to the above theorem would be to left Bousfield localise
the relative projective model structure at the class of homology isomorphisms.
The key step for that approach would be to find a set of maps $S$ such that the
$S$-equivalences are the homology isomorphisms. Since this is essentially
done for us by Smith's theorem, the alternate approach is unlikely to be quicker.

Recall the injective model structure of \cite{gre99} on $d \acal$ which we write
as $d \acal_i$. The cofibrations are the monomorphisms and the weak equivalences
are the quasi-isomorphisms.

\begin{corollary}
The identity functor from $d \acal_{dual}$
to $d \acal_i$ is the left adjoint of
a Quillen equivalence.
\[
\id
:
d \acal_{{dual}}
\adjunct
d \acal_{i}
:
\id
\]
\end{corollary}
\begin{proof}
The generating cofibrations of $d \acal_{{dual}}$ are monomorphisms
and the weak equivalences are exactly the homology isomorphisms.
\end{proof}

\begin{lemma}\label{lem:LRquillen}
There is a symmetric monoidal Quillen pair
\[
L
:
\ch(\qq)
\adjunct
d \acal_{dual}
:
R
\]
where $L V = S^0 \otimes V$ and
$R A = \acal(S^0, A)_*$. Thus, $d \acal_{dual}$ is a closed
$\ch(\qq)$--model category.
Moreover, if we let $[-,-]_*^{\acal}$ denote maps
in the homotopy category of $d \acal_{dual}$
and assume that $A$ is cofibrant and $B$ is fibrant, then we have a natural isomorphism
\[
[A,B]_*^{\acal} \cong \h_*(\acal( A,  B)_*).
\]
\end{lemma}
\begin{proof}
It is routine to check that $(L,R)$ are a symmetric monoidal Quillen pair.
The statement about maps in the homotopy category follows from the
enrichment in $\ch(\qq)$.
\end{proof}

More generally,
if $i \co V \to V'$ is a cofibration
in $\ch(\qq)$ and $P$ is a dualisable object of $\acal$, then
$i \otimes P$ is a cofibration of $d \acal_{dual}$.
While we do not have explicit generating sets for the dualisable
model structure, it is quite well--behaved, as the following
two results show.

\begin{lemma}
If $(\theta, \phi)$ is a cofibration of $d \acal_{dual}$,
then $\theta$ (the map on the nubs) is a cofibration of
$d \ocal_\fcal \leftmod$.
If $A \in d \acal_{dual}$ is cofibrant, then
$A \otimes -$ preserves weak equivalences.
If $(\theta, \phi)$ is a fibration in $d \acal_{dual}$
then $\theta$ and $\phi$ are surjective.
\end{lemma}
\begin{proof}
We claim that if $P$ is the nub of a dualisable
object of $\acal$, then it is cofibrant as an object of
$d \ocal_\fcal \leftmod$.
Consider a lifting problem comparing $0 \to P$
with an acyclic fibration $f$ of
$d \ocal_\fcal \leftmod$.
We know that
$\hom_{\ocal_\fcal}(P,-) \cong DP \otimes -$
is an exact functor, so $\hom_{\ocal_\fcal}(P,f)$
is an acyclic fibration. Thus we can solve
the lifting problem and our claim is true.
Let $i_n \otimes P \co S^{n-1} \otimes P \to D^n \otimes P$,
be a generating cofibration of $d \acal_{dual}$.
Then since the nubs of the domain and codomain are cofibrant
$\ocal_\fcal$--modules, the map induced by $i_n \otimes P$ on nubs
is a cofibration of $\ocal_\fcal$--modules.
It follows that every every cofibration of $d \acal_{dual}$
is a cofibration on the nubs.

The second statement follows from the fact that
if $N$ is a cofibrant object of
$d \ocal_\fcal \leftmod$, then $- \otimes N$
preserves homology isomorphisms.

The last statement follows from the fact that
for any dualisable $P$, the map $0 \to D^n \otimes P$
is a cofibration of $d \acal_{dual}$ and is a homology isomorphism.
Hence any fibration $f$ has the right lifting property with respect to this map.
By the discussion after Definition
\ref{def:widesphere}, it follows that the components of $f$ must be surjective.
\end{proof}

\begin{theorem}\label{thm:monoidal}
The category $d \acal$, equipped with the dualisable model structure,
is a proper symmetric monoidal
model category that satisfies the monoid axiom.
Moreover, the monoidal product on $\ho(d \acal_{dual})$
is compatible with the smash product of rational $\torus$--equivariant spectra.
\end{theorem}
\begin{proof}
Since the cofibrations are contained in the monomorphisms,
left properness follows from the left properness
of $\ch(\qq)$ and $\ocal_\fcal \leftmod$.
For right properness take some pullback
diagram of a quasi-isomorphism along a fibration
in $d \acal$:
\[
\xymatrix{A \ar[r]^{\simeq} & B & \ar@{->>}[l] C}.
\]
The pullback of this diagram in $d \acal$ is given by applying $\Gamma_h$
to pullback of the induced diagram in $d \hat{\acal}$.
The functor $\ecal^{-1} \ocal_\fcal \otimes_\qq -$ commutes with pullbacks
(where one leg is a surjection) so it follows that the pullback in
$d \hat{\acal}$ is given by the objectwise pullback $D$.
Furthermore, the fibrations of $d \acal$ are surjections,
so $D \to C$ is a quasi-isomorphism in $d \hat{\acal}$.
The functor $\Gamma_h$ preserves quasi--isomorphisms by
\cite[Proposition 20.3.4]{gre99}, so $\Gamma_h D \to C$
is a quasi--isomorphism as desired.

To prove the pushout product axiom we note that the unit is cofibrant
and the pushout of two cofibrations is again a cofibration
by the pushout product axiom for $\ch(\qq)$. Now consider
the pushout of an acyclic cofibration and a generating
cofibration. It is routine to check that the domain and
codomain of the pushout product both have trivial homology,
hence the map is a weak equivalence.

To prove the monoid axiom, note that for any generating cofibration
$i$ and any $A \in d \acal$ the map
$i \otimes A$ is a monomorphism. It follows that
for an acyclic cofibration $j$, $j \otimes A$ is a monomorphism.
By Barwick \cite[Corollary 2.7]{barwick10}
we may assume that the domains of the generating acyclic
cofibrations are cofibrant. Hence the cofibre $Cj$ of
a generating acyclic
cofibration $j$ is both cofibrant and acyclic.
Since $Cj$ is
cofibrant, $Cj \otimes A$ is weakly equivalent to
$Cj \otimes QA$ for $QA$ a cofibrant replacement
of $A$. But $Cj$ is acyclic, so
$Cj \otimes QA$ and hence $Cj \otimes A$ are acyclic.
Thus any map of the form $j \otimes A$
is a monomorphism and a quasi--isomorphism.
Such maps are closed under pushouts and
transfinite compositions, so $d \acal_{dual}$ satisfies the
monoid axiom.

To prove the last statement, recall
Greenlees's short exact sequence from
Theorem \ref{thm:greenleesmonoidal}.
That results relates the smash product of rational
$\torus$--spectra to the tensor product in $\acal$.
The ad-hoc construction of $\tor$ (and hence the derived monoidal product) is defined using the wide spheres. Since the wide spheres are dualisable the derived monoidal product on $\ho(d \acal_{dual})$
is compatible with that short exact sequence.
\end{proof}

Thus we have completed our task of finding a monoidal model structure on
$d \acal$ which is Quillen equivalent to the injective
model structure of \cite{gre99} and has the correct monoidal product.

We can also use the above results to put a monoidal
model structure on $d \hat{\acal}$. This is in fact quite instructive,
as we seem to need a larger class than the dualisable objects in this case.
Indeed, we actually construct a model structure using flat objects of
$d \hat{\acal}$ which are not dualisable. This shows that
even in quite reasonable categories it is not always
possible to construct a dualisable model structure.
See \cite[Remark 6.11]{brfranke} for a discussion
of when a dualisable model structure is likely to exist.

\begin{rmk}\label{rmk:flatmodel}
The analogues of Proposition \ref{prop:smithmodel} and
Theorem \ref{thm:modelstructure} hold for
$d \hat{\acal}$. The extra work required is to find a set of generators.
In this case a set of generators is given by the set of objects of the form
\[
0 \to \ecal^{-1} \ocal_\fcal \otimes V \quad \textrm{and} \quad
\Sigma^k \ocal_\fcal \to \ecal^{-1} \ocal_\fcal \otimes V
\]
for $V$ a finite dimensional
graded vector space without differential and $k \in \zz$.
The first kind allows us to `hit' any element of the vertex.
To `hit' an element $n$ (of degree $k$) in the nub
of $A=(N \to \ocal_\fcal \otimes U)$,
let $U'$ be the vector subspace of $U$ generated by the
$u_i$ that occur in $\beta(n) = \sum_i \sigma_i \otimes u_i$.
We define
$B=(\Sigma^k \ocal_\fcal \to \ecal^{-1} \ocal_\fcal \otimes V)$
by sending the element 1 to $\sum_i \sigma_i \otimes u_i$
in $\ecal^{-1} \ocal_\fcal \otimes U'$.
We then have a map from $A$ to $B$ defined by sending
$1$ to $n$ on the nubs and using the inclusion for the vertices.
If we enlarge this set of generators to include the set
$\pcal$ of representatives of isomorphism classes of dualisable objects of $\acal$,
then the identity is a left Quillen functor from
$d \acal_{dual}$ to $d \hat{\acal}$.

Note that the wide spheres are not sufficient, as there are objects of
$d \hat{\acal}$ whose structure map is not surjective (even
after inverting $\ecal$). Indeed, we seem to need generators of the form
$0 \to \ecal^{-1} \ocal_\fcal \otimes V $ to hit every element of the vertex
(and these generators are not in $\acal$ and are not dualisable
in $\hat{\acal}$). That is, there are not enough
dualisable objects in this category.
Since the nub and vertex of these generators are projective, it follows that
this model structure is monoidal. Hence we call this the \textbf{flat model structure}
on $d \hat{\acal}$.
\end{rmk}

\section{The Quillen equivalence}\label{sec:quillen}

In this section we to construct a symmetric monoidal
Quillen equivalence between $d \acal_{dual}$ and a suitable
model structure on $d R_\bullet \leftmod$, see  Theorem \ref{thm:monoidalequivalence}.
This gives a choice of algebraic model for rational $\torus$--spectra,
with both categories having some advantages, as discussed in the introduction.
We note that a Quillen equivalence between $d \acal_{dual}$ and
$d R_\bullet \leftmod$ is given by Greenlees and Shipley in \cite[Propositions 16.5 and 17.8]{greshi}.
However our proofs take into account the monoidal structure, are more explicit
and are more algebraic in nature.

\begin{proposition}
There is a proper and cellular model structure on $d R_\bullet \leftmod$ with
cofibrations and weak equivalences defined objectwise
(using the projective model structures
of  $\ocal_\fcal \leftmod$ and $\ch(\qq)$).
If we equip $d R_\bullet \leftmod$ with
this model structure
there is a Quillen pair
\[
\inc
:
d \acal_{dual}
\adjunct
d R_\bullet \leftmod
:
\Gamma.
\]
Furthermore, the model category $d R_\bullet \leftmod$ is a symmetric monoidal model category
and $\inc$ is a monoidal functor.
\end{proposition}
\begin{proof}
It is a standard task to check that the model structure on $d R_\bullet \leftmod$ exists, is proper
and is cellular.
Full details in a much more general setting are given by
Greenlees and Shipley in \cite[Section 3]{gsmodulediagram}.
The adjunction is a Quillen pair since $\inc$ preserves cofibrations and
homology isomorphisms.
\end{proof}

We can give the generating sets for this model structure.
For each $n \in \zz$, let $i_n \co S^{n-1} \to D^n$ in $\ch(\qq)$ be the inclusion.
We also temporarily let $R= \ocal_\fcal$ and $T= \ecal^{-1} \ocal_\fcal$.
The generating cofibrations, denoted $I_{R_\bullet}$, are given by the following maps, for each $n \in \zz$.
\[
\begin{array}{rcccc}
(i_n \otimes R, \id, \id) & \co &
(S^{n-1} \otimes R, i_n \otimes T, D^{n-1} \otimes T, 0, 0) & \longrightarrow & (D^{n-1} \otimes R, \id, D^{n-1} \otimes T, 0, 0) \\
(\id,i_n \otimes T , \id) & \co & (0, 0, S^{n-1} \otimes T, 0, 0) & \longrightarrow & (0, 0, D^n \otimes T, 0, 0) \\
(\id, \id, i_n ) & \co & (0, 0, D^n \otimes T, i_n \otimes T, S^{n-1}) & \longrightarrow & (0, 0, D^n \otimes T, \id, D^n).
\end{array}
\]
The set of generating acyclic cofibrations, denoted $J_{R_\bullet}$, is
given by the following collection of maps for $n \in \zz$.
\[
\begin{array}{rcccc}
(0, \id, \id) & \co & (0,0, D^n \otimes T, 0, 0) & \longrightarrow & (D^n \otimes R, \id, D^n \otimes T, 0, 0) \\
(\id, 0, \id) & \co & (0, 0, 0, 0, 0) & \longrightarrow & (0, 0, D^n \otimes T, 0, 0) \\
(\id, \id, 0) & \co & (0, 0, D^n \otimes T, 0, 0) & \longrightarrow & (0, 0, D^n \otimes T, \id, D^n).
\end{array}
\]

We would like to make the Quillen pair between
$d \acal$ and  $d R_\bullet \leftmod$
into a Quillen equivalence.
We do so by right Bousfield localising $d R_\bullet \leftmod$.
A comprehensive account of right Bousfield localisations can
be found in Hirschhorn \cite{hir03}.
However we will primarily use work of the author
and Roitzheim \cite{brlocal}
since we are in a stable setting and are interested in the monoidal properties
of the localisation.
We first need to give a set of cells, these will
determine the new weak equivalences
of $d R_\bullet \leftmod$. Moreover, every object of the
new homotopy category will be built from
these cells via homotopy colimits \cite[Theorem 5.1.5]{hir03}.

As we want to make $(\inc, \Gamma)$ into a Quillen equivalence
we look for our cells in $d \acal$.
Since $\acal$ models rational $\torus$--spectra,
we know a set of objects that detects weak equivalences
in $d \acal_{dual}$: the objects  $\pi_*^{\acal}(\torus/H_+)$
for $H$ a closed subgroup of $\torus$. By \cite[Lemma 13.6]{greshi}
we can construct each such object from $S^0$ using cofibre
sequences and suspensions by
functions with finite support. Hence we
have the following definition.

\begin{definition}
We define the set of \textbf{cells} $K$ to be the set
of all shifts of algebraic spheres $\inc S^\nu$ and $\inc S^{-\nu}$
for $\nu \co \fcal \to \zz_{\geqslant 0}$ of finite support,
see Definition \ref{def:algshere}.
\[
K = \{
\inc S^{n +\nu}, \inc S^{n-\nu} \ \mid \ n \in \zz, \
\nu \co \fcal \to \zz_{\geqslant 0} \textrm{ with finite support}
\}
\]
\end{definition}

Note that the set $K$ consists of cofibrant objects of $d R_\bullet \leftmod$
and that $K$ is closed under the tensor product. In the language of
Barnes and Roitzheim \cite{brlocal} this set is a monoidal and stable set of cells.
We give a proof that this set detects weak equivalences in
$d \acal_{dual}$ that does not require spectra. In the language of triangulated
categories, the following result says that the algebraic spheres generate
$\ho(d \acal_{dual})$.

\begin{proposition}\label{prop:repgenerate}
An object $A$ of $d \acal_{dual}$ is weakly equivalent to $0$ if and only if
$[S^\nu, A]_*^{\acal}=0$ for all algebraic spheres $S^\nu$.
\end{proposition}
\begin{proof}
The algebraic spheres $S^\nu$ are dualisable in $\acal$
and hence are cofibrant in $d \acal_{dual}$.
Thus $[S^\nu, A]_*^{\acal}=0$ if and only if
$\acal(S^\nu, \fibrep A)_*$ is an acyclic chain complex,
where $\fibrep$ denotes fibrant replacement
in the dualisable model structure.
Let $\fibrep A=(\beta \co N \to \ecal^{-1} \ocal_\fcal \otimes U)$
and take $u$ to be a cycle of degree $k$ in $U$.
Since $\ecal^{-1} \beta$ is an isomorphism, there is an
Euler class $c^\nu$ such that $1 \otimes u$ is the image of
some cycle $x$ under
\[
c^{-\nu} \circ \beta \co \Sigma^\nu N \to \ecal^{-1} \ocal_\fcal \otimes U.
\]
Define a map $\Sigma^k \ocal_\fcal \to \Sigma^\nu N$ by sending
$1$ to $x$ and define a map $\Sigma^k \qq \to U$ by sending
the generator to $u$.
This gives a cycle map $\Sigma^k S^0 \to \Sigma^\nu \fibrep A$.
Hence we have a cycle map $\Sigma^k S^{-\nu} \to \fibrep A$.
This cycle map is a boundary as $\acal(S^V, \fibrep A)_* \simeq 0$.
So $u$ is a boundary and $\h_*(U)=0$.

Let $Z = (N \to 0)$ then it follows that $\fibrep A \to Z$
is a quasi-isomorphism.
Moreover, $Z$ is fibrant in $d \acal_{dual}$.
The functor  $\acal(S^\nu, -)$ is right adjoint
to the Quillen functor $S^\nu \otimes -$
from $\ch(\qq)$ to $d \acal$.
Hence it preserves all weak equivalences between fibrant objects
and $\acal(S^\nu, Z)_* \simeq \acal(S^\nu, \fibrep A)_* \simeq 0$.

We now show that $Z$ is acyclic.
Take some cycle $n \in N$ of degree $k$.
We can map $1 \in \Sigma^k \ocal_\fcal$ to $n$.
This defines a cycle map $S^k \to Z$. Hence this map
is a boundary and hence so is $n$. It follows that $Z$ and $\fibrep A$
are quasi-isomorphic to zero as desired.
\end{proof}

We can now describe the right Bousfield localisation of
$d R_\bullet \leftmod$ at the set of cells $K$.

\begin{theorem}\label{thm:cellmodel}
There is a stable monoidal model structure on $d R_\bullet \leftmod$
whose weak equivalences are those maps $f \co A \to B$ such that
\[
[S^\nu, f]_*^{R_\bullet} \co [S^\nu, A]_*^{R_\bullet} \longrightarrow [S^\nu, B]_*^{R_\bullet}
\]
is an isomorphism (of sets of maps in the homotopy category of
$d R_\bullet \leftmod$) for all algebraic spheres $S^\nu$..
Furthermore, this model structure is proper and cellular.
The set of generating acyclic cofibrations is given by $J_{R_\bullet}$.
The set of generating cofibrations is given by the union of $J_{R_\bullet}$
with the set of maps
$i_n \otimes \inc S^\nu \co S^{n-1} \otimes \inc S^\nu \to D^n \otimes \inc S^\nu$
where $n \in \zz$ and $\nu \co \fcal \to \zz_{\zz \geqslant 0}$
with finite support.
We write $K \cell d R_\bullet \leftmod$ for this model structure.
\end{theorem}
\begin{proof}
See \cite[Theorems 5.9 and 7.2]{brlocal}.
\end{proof}

Our next task is to examine the weak equivalences a bit more carefully.
Given a map $f \co X \to Y$ in $d R_\bullet \leftmod$
we define $\fibrep f$ to be a map which makes the following square commute,
where $X \to \fibrep X$ and $Y \to \fibrep Y$ are fibrant replacements of $X$ and $Y$.
\[
\xymatrix{
X \ar[r]^f \ar[d]^-{\simeq} &
Y \ar[d]^-{\simeq} \\
\fibrep X \ar[r]^{\fibrep f} &
\fibrep Y \\
}
\]

\begin{lemma}\label{lem:gammaequiv}
A map $f$ is a weak equivalence in $K \cell d R_\bullet \leftmod$
if and only if $\Gamma \fibrep f$ is a quasi-isomorphism in $d \acal$.
We call such maps \textbf{$\Gamma$-equivalences}.
\end{lemma}
\begin{proof}
Since each cell $S^\nu$ is in the image of the functor $\inc$,
it follows that a map $f$ is a weak equivalence if and only if
\[
[S^\nu, \Gamma \fibrep f]_*^{\acal}
\co
[S^\nu, \Gamma \fibrep A]_*^{\acal}
\longrightarrow
[S^\nu, \Gamma \fibrep B]_*^{\acal}
\]
is an isomorphism (of maps in the homotopy category of $d \acal_{dual}$).
Let $Z$ be the homotopy fibre of $\Gamma \fibrep f$ (which is fibrant).
Then, by Proposition \ref{prop:repgenerate}, $Z$ is quasi-isomorphic to 0
if and only if $f$ is a weak equivalence of
$K \cell d R_\bullet \leftmod$. The result follows immediately.
\end{proof}

We can now give the main theorem of this section.
Note that while the algebraic spheres are the
correct set of cells to use at the level of homotopy,
they would not be sufficient for the purposes of
Theorem \ref{thm:modelstructure}.

\begin{theorem}\label{thm:monoidalequivalence}
There is a commutative diagram of Quillen pairs as below, with left adjoints
on top.
\[
\xymatrix@C+1.8cm@R+0.4cm{
& d \acal_{dual}
\ar@<-3pt>[dl]_{\inc}
\ar@<+3pt>[dr]^{\inc}
\\
K \cell d R_\bullet \leftmod
\ar@<+3pt>[rr]^{\id}
\ar@<-3pt>[ur]_{\Gamma}
&&
d R_\bullet \leftmod
\ar@<+3pt>[ll]^{\id}
\ar@<+3pt>[ul]^{\Gamma}
}
\]
Furthermore, the Quillen adjunction $(\inc, \Gamma)$ between $d \acal_{dual}$
and $K \cell d R_\bullet \leftmod$ is a symmetric monoidal Quillen equivalence.
\end{theorem}
\begin{proof}
We first need to show that $\inc$ is a left Quillen functor from
$d \acal_{dual}$ to $K \cell d R_\bullet \leftmod$.
By \cite[Theorem 3.1.6]{hir03} it suffices to
show that for any cofibrant $A \in d \acal_{dual}$,
and any weak-equivalence $f$ of $K \cell d R_\bullet \leftmod$
the map $[\inc A, f]_*^{R_\bullet}$ is an isomorphism.
By adjointness, $[\inc A, f]_*^{R_\bullet}$ is isomorphic to
$[A, \Gamma \fibrep f]_*^{\acal}$.
The morphism $\Gamma \fibrep f$ is a quasi-ismorphism
of $d \acal_{dual}$ by Lemma \ref{lem:gammaequiv},
hence $[\inc A, f]_*^{R_\bullet}$ is an isomorphism as desired.
Furthermore,  the functor $\inc$ is clearly symmetric monoidal.

To show that $(\inc, \Gamma)$ is a Quillen equivalence,
we consider the derived unit and counit.
Start with some cofibrant $A \in d \acal_{dual}$ and let $\fibrep$
be the fibrant replacement functor of $d R_\bullet \leftmod$.
We claim the the derived unit map
\[
A \to \Gamma \inc A \to \Gamma \fibrep \inc A
\]
is a quasi-isomorphism.
We know that the first map is an isomorphism, so we focus on
$\inc A \to \fibrep \inc A$.
Let $A=(\beta \co N \to \ecal^{-1} \ocal_\fcal \otimes U)$,
and factor the map $N \to \ecal^{-1} \ocal_\fcal \otimes U$
as an acyclic cofibration $\alpha \co N \to N'$ followed by a fibration
$\beta' \co N' \to \ecal^{-1} \ocal_\fcal \otimes U$.
The map $\beta'$ induces a map
$\beta'' \co \ecal^{-1} N' \to \ecal^{-1} \ocal_\fcal \otimes U$.
A fibrant replacement of $\inc A$ is given by
$(N', \beta'', \ecal^{-1} \ocal_\fcal \otimes U, \id, U)$
and $\alpha$ induces a quasi-isomorphism from $\inc A \to \fibrep \inc A$.
If we apply $\Gamma_v$ to this construction of $\fibrep \inc A$
we obtain the object
$(\beta' \co N' \to \ecal^{-1} \ocal_\fcal \otimes U) \in d \hat{\acal}$.
Furthermore, $\alpha$ induces a quasi-isomorphism (in $d \hat{\acal}$)
from $A$ to $\Gamma_v \fibrep \inc A$.
Since $\Gamma_h$ preserves quasi-isomorphisms, our claim holds.

Now let $C$ be a fibrant object of $K \cell d R_\bullet \leftmod$
and let $\cofrep$ denote a cofibrant replacement in $d \acal_{dual}$.
We claim that the derived counit map
\[
\inc \cofrep \ \Gamma C \to \inc \Gamma C \to C
\]
is a $\Gamma$--equivalence. We first note that
since $\inc$ preserves quasi-isomorphisms and
every quasi-isomorphism is weak equivalence of
$K \cell d R_\bullet \leftmod$, the first map of the above composite
is a weak equivalence.
So it suffices to show that $\inc \Gamma C \to C$
is a $\Gamma$--equivalence. After applying the
derived functor of $\Gamma$ we obtain a commutative diagram as below.
\[
\xymatrix{
\Gamma \fibrep \inc \Gamma C
\ar[r]  & \Gamma \fibrep C  \\
\Gamma \inc \Gamma C \ar[r] \ar[u]^{\simeq} &
\Gamma C \ar[u]^{\simeq}
}
\]
The left vertical map is a weak equivalence by the arguments
used in showing that the derived unit is a weak equivalence.
The right hand vertical is a weak equivalence as $C$ is already fibrant.
Finally the lower horizontal map is an isomorphism as $\Gamma \inc \cong \id$.
Hence the derived counit map is a weak equivalence.
\end{proof}

We can also phrase the above Quillen equivalence in terms of an
inclusion of triangulated subcategories.
That is,  $\acal = \ho (d \acal_{dual})$ is the
smallest full triangulated subcategory
of $\ho (R_\bullet \leftmod)$ that is closed under coproducts and contains the cells $K$.
To show this we combine the following proposition,
Proposition \ref{prop:repgenerate} and
\cite[Theorem 9.3]{brlocal} (with spectra replaced by
$\ch(\qq)$).
Hence we are fully justified in calling
$R_\bullet \leftmod$ a larger category than $d \acal_{dual}$.

We say that an object $A$ of a stable model category
$\ccal$ is \textbf{homotopically compact} if
given any collection of objects $B_i$ the natural map
\[
\xymatrix{\oplus_i [A,B_i]_*^\ccal \ar[r] & [A,\coprod_i B_i]_*^\ccal}
\]
is an isomorphism of sets of maps in the homotopy category of $\ccal$.

\begin{proposition}\label{prop:hocompact}
The shifted algebraic spheres $S^{n +\nu}$
are homotopically compact in $d \acal_{dual}$ and
$K \cell d R_\bullet \leftmod$.
\end{proposition}
\begin{proof}
Since these two model categories are Quillen equivalent,
it suffices to show the cells of
$K \cell d R_\bullet \leftmod$ are homotopically compact.
An object $(M, \alpha, N, \gamma, U)$ of $K \cell d R_\bullet \leftmod$
is fibrant if and only if the structure maps $\alpha$ and
$\gamma$ are surjective.
It follows that if $B_i$ is a collection of fibrant objects
of $K \cell d R_\bullet \leftmod$ then $\oplus_i B_i$ is fibrant.
Hence we have an isomorphism
\[
\xymatrix{
[S^V,\oplus_i B_i]_*^{K \cell R_\bullet}
\ar[r] &
\h_* \left( \acal(S^V, \oplus_i B_i)_* \right)
}\]
It is easily seen that an $\ocal_\fcal$--module map from $\ocal_\fcal(V)$
into an infinite direct sum lands in some finite sum.
The same is true for $\ecal^{-1} \ocal_\fcal$ in
$\ecal^{-1} \ocal_\fcal$--modules and $\qq$ in
$\ch(\qq)$. So it follows that the natural map
\[
\xymatrix{ \oplus_i \acal(S^V, B_i)_*
\ar[r] & \acal(S^V, \oplus_i B_i)_* }
\]
is an isomorphism and the result follows immediately.
\end{proof}

As well as $K \cell R_\bullet \leftmod$ and $d \acal_{dual}$
we would also like a model structure on $d \hat{\acal}$
and a Quillen equivalence to either of these two model categories.
An adaptation of our earlier work provides such a model structure.

\begin{theorem}
There is a model structure on
$d \hat{\acal}$ where the weak equivalences are
those maps $f$ such that $\h_* (\Gamma_h f)$ is an isomorphism
and the generating cofibrations are the maps $\inc (i_n \otimes P)$ for $P \in \pcal$
and $n \in \zz$.
Moreover the Quillen pair $(\inc, \Gamma)h)$ is a
Quillen equivalence between this model structure on $d \hat{\acal}$
and $d \acal_{dual}$.
\end{theorem}
\begin{proof}
By Remark \ref{rmk:flatmodel} $d \hat{\acal}$ the technical conditions in the proof
of Proposition \ref{prop:smithmodel} hold for the category $d \hat{\acal}$.
The existence of the model structure then follows the same
method as Theorem \ref{thm:modelstructure}, except
that the weak equivalences are those
maps $f$ such that $\h_* (\Gamma_h f)$ is an isomorphism.
This is possible since
$\Gamma_h$ commutes with filtered colimits.
An analogous argument to Theorem \ref{thm:monoidalequivalence}
provides the Quillen equivalence.
\end{proof}

\begin{conjecture}\label{conj:torusmodel}
It should be possible to extend these results
to the case of the product of $r$ copies of $\torus$, $r >1$.
The algebraic model for
$\torus^r$--equivariant rational spectra
is defined in \cite{gre08torus}. In \cite{greshi} the algebraic model is
given an injective model structure where the cofibrations
are the monomorphisms and the weak equivalences are the homology isomorphisms.
This model structure is not monoidal for the
same reason that $d \acal_i$ is not. Thus if we want to study
the monoidal properties of
rational $\torus^r$--spectra we need
an analogue of the dualisable model structure.

The key steps to generalising this section to
$\torus^r$ are showing that one has the analogue of
wide spheres  (which form a set of
categorical generators) and algebraic spheres (which
generate the homotopy category).
We leave this for future work since
the algebraic model for rational $\torus^r$--spectra
is much more complicated to define and constructing these two
collections of spheres would be a very substantial task.
\end{conjecture}

\addcontentsline{toc}{part}{Bibliography}
\bibliography{bibdual}
\bibliographystyle{alpha}

\end{document}